\documentclass[a4paper,12pt,reqno]{amsart}
\usepackage{amssymb,amsthm}
\usepackage{ifthen}
\usepackage{graphicx}
\usepackage{float}
\usepackage{tcolorbox}
\usepackage{diagbox}
\theoremstyle{plain}

\newtheorem{thm}{Theorem}[section]
\newtheorem*{thm*}{Theorem}

\numberwithin{equation}{section}

\newtheorem{cor}{Corollary}[section]

\newtheorem{defn}{Definition}[section]

\theoremstyle{definition}
\newcounter {own}
\def\theown {\thesection  .\arabic{own}}

{\qed\bigskip}

\newcounter{alphabet}
\newcounter{tmp}

\newcommand{\ds}{\displaystyle}
\newcounter{minutes}
\divide\time by 60
\newcounter{hours}
\multiply\time by 60 \addtocounter{minutes}{-\time}
\begin{document}
\bibliographystyle{amsplain}
\title{Sharp results on sampling with derivatives in  bandlimited functions}

%=========================================================================
\thanks{%$^\dagger$
File:~\jobname .tex,
          printed: 2013-07-30,
          \thehours.\ifnum\theminutes<10{0}\fi\theminutes}
%=========================================================================

\author{A. Antony Selvan}

\address{A. Antony Selvan,
Indian Institute of Technology Dhanbad, Dhanbad-826004, India.}
\email{antonyaans@gmail.com}
%
%\author{R. Radha $^\dagger$}
%\address{R. Radha, Department of Mathematics,
%Indian Institute of Technology Madras, Chennai--600 036, India.}
%\email{radharam@iitm.ac.in}
%\subjclass[2000]{Primary  42C15, 94A20}
%\keywords{ Bernstein's inequality, frames, Meyer scaling function, nonuniform sampling, Riesz basis, Wirtinger's inequality.\\
%$^\dagger$ {\tt Corresponding author}
%}
\maketitle
\pagestyle{myheadings}
\markboth{A. Antony Selvan}{Sharp results on sampling with derivatives in  bandlimited functions }
\begin{abstract}
We discuss the problems of uniqueness, sampling and reconstruction with derivatives in the space of bandlimited functions.
We prove that if $X=\{x_i:i\in\mathbb{Z}\}$, $\cdots<x_{i-1}<x_i<x_{i+1}<\cdots$, is a separated 
sequence of real numbers such that
$\delta:=\sup_{i}(x_{i+1}-x_i)<\dfrac{\nu_k}{\sigma}$, then any bandlimited function of bandwidth $\sigma$ can be reconstructed  uniquely  and stably from
its nonuniform samples $\{f^{(j)}(x_i):j=0,1,\dots, k-1, i\in\mathbb{Z}\}$, where $\nu_k$ denotes the Wirtinger-Cimmino constant.
We also prove that if $\delta\leq \dfrac{\nu_k}{\sigma}$, then $X$ is a set of uniqueness for the space of bandlimited functions of bandwidth $\sigma$ when the samples involving the first $k-1$ derivatives.
As a by-product, we obtain  the sharp the maximum gap condition for samples involving first derivative. \\\\

\vspace{3mm}
\end{abstract}

%%%%%%%%%%%%%%%%%%%%%%%%%%%%%%%%%%%%%%%%%%%%%%%%%%
%
% Finally, you can write your article
%
\section{Introduction}
Let $\mathcal{B}_\sigma$ denote  the space of
entire functions  of exponential type $\leq\sigma$ that are square integrable on the real axis. 
The classical Shannon's sampling theorem states that every $f\in\mathcal{B}_\sigma$ can be reconstructed from the sampling formula
\begin{eqnarray*}
f(x)=\ds\sum_{n\in\mathbb{Z}}f\left(\dfrac{n\pi}{\sigma}\right)\dfrac{\sin \sigma\left(x-n\pi/\sigma\right)}{\sigma \left(x-n\pi/\sigma\right)}.
\end{eqnarray*}
In many applications, such as aircraft instrument communications, air traffic control simulation, or telemetry \cite{Fogel}, one
can consider the possibility of obtaining sampling expansion which involved sample values of a function and its derivatives.
Fogel \cite{Fogel}, Jagermann and Fogel \cite{JagermanFogel}, Linden and Abramson \cite{LindenAbramson} extended the Shannon sampling theorem  in this direction. They proved that
if the values of $f$ and its first $R-1$ derivatives are known on the sequence $\left\{rn\pi/\sigma:n\in\mathbb{Z}\right\}$, then every $f\in \mathcal{B}_\sigma$ can be reconstructed from  the sampling formula
\begin{eqnarray*}
f(x)=\ds\sum_{n\in\mathbb{Z}}\sum_{k=0}^{R-1}
f_k\left(x_n\right)\dfrac{(x-x_n)^k}{k!}\left[\dfrac{\sin \dfrac{\sigma}{R}\left(x-x_n\right)}{\dfrac{\sigma}{R} \left(x-x_n\right)}\right]^R,
\end{eqnarray*}
where $x_n=rn\pi/\sigma$ and $f_k(x_n)$ are linear combinations of $f(x_n), f^\prime(x_n),\dots,f^{(k)}(x_n):$
$$f_k(x_n):=\sum_{i=0}^k {k\choose i}(\sigma/R)^{k-i}\left[\dfrac{d^{k-i}}{dx^{k-i}}\left(\dfrac{x}{\sin x}\right)^R\right]\Bigg|_{x=0}f^{(i)}(x_n).$$
The problem of nonuniform sampling expansion involving derivatives is well studied in the literature.
Rawn  in \cite{Rawn} extended the Kadec $1/4-$theorem for sampling involving derivatives. He proved that if $\{x_n:n\in\mathbb{Z}\}$ is a sampling sequence belonging  to the class
$S=\left\{(x_n):|x_n-n|\leq d< 1/{4R},~\forall ~n\in\mathbb{Z}\right\}$, then any $f\in\mathcal{B}_\sigma$ can be recovered from its samples $\{f^{(j)}(x_n):j=0,1,\dots, R-1, n\in\mathbb{Z}\}$.
The general nonuniform sampling problem involving derivatives was studied by Gr\"{o}chenig and others in terms of Beurling density. They proved that
if the Beurling-Landau density, $D(X)$ of the set $X$, is greater than $\tfrac{k\pi}{\sigma}$, then any $\sigma$-bandlimited function $f$ can be reconstructed uniquely and stably from its sample
values $\{f^{(j)}(x_n):j=0,1,\dots, k-1, n\in\mathbb{Z}\}$. For further details, we refer the reader to \cite{Beurling, Grochenig2, Landau}.

The uniqueness problem  is studied in the mathematical literature in the context of separation of zeros \cite{Cluine,Walker}.
In this paper, we examine the separation of real zeros of multiplicity $k$ for the bandlimited functions.
For a non-zero entire funtion $f$, we define the longest zero-free interval of $f$ of order $k-1\geq0$ as follows:
$$M_k(f):=\sup\left\{b-a:f^{(l)}(x)\neq 0,x\in(a,b),l=0,1,\dots,k-1\right\}.$$
If all the real zeros of $f$ are bounded on the left or right, then $M(f)=\infty$. Otherwise, the zeros of $f$ can be arranged in a doubly infinite sequence $\{x_n:n\in\mathbb{Z}\}$ with $x_n\leq x_{n+1}$.
In this case,
$$M_k(f):=\sup\left\{x_{n+1}-x_n:f^{(l)}(x_n)=0,n\in\mathbb{Z},l=0,1,\dots,k-1\right\}.$$
We now define $$d_k:=\inf\left\{M_k(f):0\neq f\in\mathcal{B}_\sigma\right\}.$$  Walker in \cite{Walker} proved that $d_1=\frac{\pi}{\sigma}$. In this paper, we prove that $d_2=\frac{2\pi}{\sigma}$ and   
hence we obtain a sharp result on uniqueness theorem involving first derivative for the bandlimited functions. We also provide lower and upper bounds for the constants $d_k, ~k\geq 3$.

The numerical aspects of nonuniform sampling expansion involving derivatives was studied by Gr\"{o}chenig in \cite{Grochenig1}. 
If $f^{(j)}(x),~j=0,\dots,k-1$ are sampled at uniform  rate, it follows from Beurling density theorem that
$f$ can be reconstructed if the uniform gap is less than $k\pi/\sigma$. From the above discussion, it is \textit{conjectured} in \cite{Adcock, Grochenig1, Raza} that 
if $\sup_{i}(x_{i+1}-x_i)<\tfrac{k\pi}{\sigma}$, then any bandlimited function  can be reconstructed  uniquely  and stably from
its nonuniform samples $\{f^{(j)}(x_i):j=0,1,\dots, k-1, i\in\mathbb{Z}\}$. Gr\"{o}chenig in \cite{Grochenig1} proved that if $\sup_{i}(x_{i+1}-x_i)<\tfrac{\pi}{\sigma}$, then any bandlimited function  can be reconstructed  uniquely  and stably from its nonuniform samples $\{f(x_i) :i\in\mathbb{Z}\}$.
In \cite{Raza}, Razafinjatovo obtained a frame algorithm for reconstructing a function $f\in \mathcal{B}_\sigma$ from its nonuniform samples $\{f^{(j)}(x_i):j=0,1,\dots, k-1, i\in\mathbb{Z}\}$ 
with maximum gap condition, namely  $\sup_{i}(x_{i+1}-x_i)=\delta<\frac{2}{\sigma}((k-1)!\sqrt{(2k-1)2k})^{1/k}$ using Taylor's polynomial approximation.
The authors in \cite{Adcock} improved the Razafinjatovo's result using Hermite interpolation. 
Radha and the author in \cite{AntoRad2} obtained a better bound than the one given in \cite{Adcock} using Wirtinger-Sobolev inequality.
While this above conjecture for $k\geq 2$ is still unresolved for the past twenty five years, the case $k=2$ is settled in this paper as follows.
\begin{thm}\label{thm1.1}
 If $\{x_i:i\in\mathbb{Z}\}$ is a separated set such that $\sup_{i}(x_{i+1}-x_i)<\tfrac{2\pi}{\sigma}$, then there exists a Bessel sequence $\left\{g_{1,i}, g_{2,i}:i\in\mathbb{Z}\right\}$ in $\mathcal{B}_\sigma$ such that
 $$f^{\prime}(x)=\sum\limits_{i\in\mathbb{Z}} f(x_i)g_{1,i}(x)+\sum\limits_{i\in\mathbb{Z}} f^{\prime}(x_i)g_{2,i}(x),$$
 for every $f\in \mathcal{B}_\sigma$. The function $f$ can be computed using the relation $$f(x)=\int\limits_{x_i}^x f^{\prime}(t)\mathrm{d}t+f(x_i).$$
\end{thm}

The paper is organized as follows. In section 2, we discuss some preliminary facts needed for the later sections. In section 3, we  prove a uniqueness theorem for bandlimited function involving derivative samples.
In section 4, we provide iterative reconstruction algorithms for the recovery of bandlimited functions from its samples involving derivatives.

\section{Preliminaries}
This section provides some useful terminology and results in order to prove our main results.
Throughout this paper, we shall adopt the following notations: $\mathcal{H}$ is a separable Hilbert space and $\mathcal{L}(\mathcal{H})$ is the space of bounded operators on $\mathcal{H}$.

\begin{defn}
A sequence $\{f_n:~n\in\mathbb{Z}\}$ in $\mathcal{H}$ is called a Riesz basis if it satisfies the following conditions:
\begin{itemize}
\item [$(i)$] It is complete, \textit{i.e.}, $\overline{span}\{f_n\}=\mathcal{H}$.
\item [$(ii)$] There exist two constants $A,B>0$ such that
\begin{equation}
A\sum_{n\in\mathbb{Z}}|d_n|^2\leq\big\|\sum_{n\in\mathbb{Z}}d_n f_n\big\|^2_{\mathcal{H}}\leq
B\sum_{n\in\mathbb{Z}}|d_n|^2,
\end{equation}
for all $(d_n)\in\ell^2(\mathbb{Z})$.
\end{itemize}
\end{defn}
\begin{defn}
A sequence $\{f_n:n\in\mathbb{Z}\}$ in  $\mathcal{H}$ is called a frame if
there exist two  constants $A,B> 0$ such that for all $f \in \mathcal{H}$ we have
\begin{equation}\label{frameineq}
A||f||^2\leq \sum_n |\langle f,f_n \rangle|^2 \leq B||f||^2 .
\end{equation}
The sequence $\{f_n:n\in\mathbb{Z}\}$ is called a Bessel sequence if at least the upper bound in \eqref{frameineq} is satisfied.
The numbers $A, B$ are called the frame bounds.
\end{defn}
Every Riesz basis is a frame. 
If $\{f_n\}$ is a frame for $\mathcal{H}$ with bounds $A,B$, then  $$Sf=\sum_n \langle f,f_n \rangle f_n$$
is a positive self-adjoint invertible operator and called the frame operator associated with  $\{f_n\}$. Let $\rho=\dfrac{2}{A+B}$. 
Then $f$ can be recovered by the \textit{frame algorithm}: Set
\begin{eqnarray*}
f_0&=&0,\\
f_{n+1}&=&f_n+\rho S(f-f_n),~ n\geq 0.
\end{eqnarray*}
Then we have $\lim\limits_{n\to\infty}f_n=f$. The error estimate after $n$ iterations turns out to be
$$
\|f-f_n\|_{\mathcal{H}}\leq\left(\dfrac{B-A}{B+A}\right)^n\|f\|_\mathcal{H}.
$$
We refer the reader to \cite{AlGr2, Young} for further details.

\begin{defn}
Let $K\in\mathcal{L}(\mathcal{H})$. A sequence of vectors $\{f_n:~n\in\mathbb{Z}\}$ in $\mathcal{H}$ is said to be a $K$-{\it frame} if there
exist two constants $A,~B>0$ such that
\begin{equation}
A\|K^*f\|_\mathcal{H}^2\leq\displaystyle\sum_{n\in\mathbb{Z}}|\langle
f,f_n\rangle_\mathcal{H}|^2\leq B\|f\|_\mathcal{H}^2,
\end{equation}
for every $f\in\mathcal{H}$.  When  $K$ is the identity operator, it coincides with ordinary frames.
\end{defn}
\begin{defn}
 Let $K\in\mathcal{L}(\mathcal{H})$. A sequence of vectors $\{f_n:~n\in\mathbb{Z}\}$ in $\mathcal{H}$ is said to be an {\it atomic system} for $K$ if  the following statements hold:
\begin{itemize}
 \item[$(i)$]  $\{f_n:~n\in\mathbb{Z}\}$ is a Bessel  sequence in $\mathcal{H}$.
 \item[$(ii)$] There exists $B>0$ such that for every $f\in \mathcal{H}$ there exists $d_f=(d_n)\in\ell^2(\mathbb{Z})$ such that $\|d_f\|_2\leq B\|f\|_\mathcal{H}$ and 
 $Kf=\ds\sum_{n\in\mathbb{Z}}d_nf_n$.
\end{itemize}
\end{defn}
\begin{thm}$($\cite{Gavruta}$)$\label{thm2.1}
 A sequence of vectors $\{f_n:~n\in\mathbb{Z}\}$ in $\mathcal{H}$ is  an {\it atomic system} for $K$ if  and only if it is a K-frame for $\mathcal{H}$. In this case, there exists a Bessel sequence
 $\{g_n:n\in\mathbb{Z}\}$ such that
$$Kf=\sum_{n\in\mathbb{Z}}\langle f,g_n\rangle f_n ~and~ K^*f=\sum_{n\in\mathbb{Z}}\langle f,f_n\rangle g_n,$$ 
for every $f\in\mathcal{H}$.
\end{thm}

For an integrable function $f$, the
Fourier transform $\widehat{f}$ of $f$ is defined by
$$\widehat f (w):= \int\limits_{\mathbb{R}}  f(x) e^{-2\pi i w  x} \mathrm{d}x, \quad w \in \mathbb{R}.$$
If $f \in L^1\cap L^2(\mathbb{R}),$ one has the Plancherel
formula $||f ||_2 = ||\widehat f||_2 $. As $ L^1 \cap L^2(\mathbb{R})$
is dense in $L^2(\mathbb{R})$, the definition of Fourier transform
is extended to functions in $L^2(\mathbb{R})$.

For $\sigma>0$, let $\mathcal{B}_\sigma$ denote the space of all $\sigma$-bandlimited functions, \textit{i.e.,}
\begin{eqnarray*}
\mathcal{B}_\sigma=\left\{f\in L^2(\mathbb{R}): \textrm{supp}\widehat{f}\subseteq [-\tfrac{\sigma}{2\pi},\tfrac{\sigma}{2\pi}]\right\}.
\end{eqnarray*}
The celebrated theorem of Paley-Wiener says that $\mathcal{B}_\sigma$ coincides with the space of
entire functions  of exponential type $\leq\sigma$. Moreover, $\mathcal{B}_\sigma$ is a reproducing kernel Hilbert space with reproducing kernel $\mathcal{K}(x,y)=\tfrac{\sin\sigma\left(x-y\right)}{\sigma \left(x-y\right)}$. 
The following inequality is well-known in the literature.
\begin{thm}[Bernstein's inequality]
If $f\in \mathcal{B}_\sigma$, then $f^\prime\in \mathcal{B}_\sigma$ and
\begin{eqnarray}
\|f^{(k)}\|_2\leq \sigma^k\|f\|_2.
\end{eqnarray}
\end{thm}

%In order to prove our main results, we make use of the following terminology and some inequalities.
\begin{defn}
Let $X=\{x_{i}:i\in\mathbb{Z}\}$ be a sequence of real or complex numbers. Then
\begin{enumerate}
%\item [$(i)$]$X$ is separated by a constant $\gamma>0$ if $\inf\limits_{i\neq j} |x_i-x_j|\geq\gamma$.
%\item [$(ii)$]$\{x_{n}:n\in\mathbb{Z}\}$ is said to be a  set of uniqueness for $\mathcal{H}$ if $f(x_n)=0$ for all $n$ implies $f\equiv0$.
\item [$(i)$]  $X$ is said to be a set of uniqueness of order $k-1$ for $\mathcal{B}_\sigma$ %with respect to the weight $\{w_n\in\mathbb{R}^*:n\in\mathbb{Z}\}$
if $$f^{(l)}(x_i)=0,~ n\in\mathbb{Z}, l=0,1,\dots,k-1,$$ implies that $f\equiv 0$.
\item [$(ii)$] $X$ is said to be a  stable set of sampling of order $k-1$ for $\mathcal{B}_\sigma$ 
if there exist constants $A$, $B>0$ such that
\begin{eqnarray}\label{pap3eqn2.7}
A\| f\|_2^2&\leq&\ds\sum\limits_{i\in\mathbb{Z}}
\ds\sum\limits_{l=0}^{k-1}|f^{(l)}(x_i)|^2 \leq B\| f\|_2^2,
\end{eqnarray}
for all $f\in\mathcal{B}_\sigma$.
\item [$(iii)$]  $X$ is said to be a set of interpolation of order $k-1$ for $\mathcal{B}_\sigma$ %with respect to the weight $\{w_n\in\mathbb{R}^*:n\in\mathbb{Z}\}$
if the interpolation problem  
$$f^{(l)}(x_i)=c_{il},~ n\in\mathbb{Z},~l=0,1,\dots,k-1,$$
has a solution $f\in \mathcal{B}_\sigma$ for every square summable sequence $\{c_{il}:i\in\mathbb{Z}\},l=0,1,\dots,k-1$.
\end{enumerate}
\end{defn}
\begin{defn}
 A set $X=\{x_{i}:i\in\mathbb{Z}\}$ is said to be a  stable set of sampling of order $(m,k-1)$ for $\mathcal{B}_\sigma$ with respect to the weight $\{w_{il}\in\mathbb{R}^*:l=0,\dots,k-1,i\in\mathbb{Z}\}$
if there exist constants $A$, $B>0$ such that
\begin{eqnarray}\label{pap3eqn2.7}
A\| f^{(m)}\|_2^2&\leq&\ds\sum\limits_{i\in\mathbb{Z}}
\ds\sum\limits_{l=0}^{k-1}w_{il}|f^{(l)}(x_i)|^2 \leq B\| f\|_2^2,
\end{eqnarray}
for all $f\in\mathcal{B}_\sigma$.
\end{defn}

Let $X=\{x_n\}$ be sequence of distinct real numbers. To each $X$ and $m\in\mathbb{N}$, we associate an exponential system 
$$\mathcal{E}(X;m):=\left\{e^{2\pi ix_nt},te^{2\pi ix_nt},\dots,t^{m-1}e^{2\pi ix_nt}:x_n\in X\right\}.$$
It is well known that $\mathcal{E}\left(\dfrac{m\pi}{\sigma}\mathbb{Z},m\right)$ is a Riesz basis for $L^2[-\frac{\sigma}{2\pi},\frac{\sigma}{2\pi}]$. Arguing as in \cite{Seip}, 
we can prove that $X$ is a set of uniqueness and interpolation of order $m-1$ for $\mathcal{B}_\sigma$ if and only if $\mathcal{E}(X,m)$ is a Riesz basis for $L^2[-\frac{\sigma}{2\pi},\frac{\sigma}{2\pi}]$.
N. Levinson in \cite{Levinson} proved that the completeness of the system $\mathcal{E}(X;m)$ is unaffected if we replace finitely many points $x_n$ by the same number of points $y_n\notin X$. R. M. Young in \cite{Young}
showed that if $\{f_n\}$ is a Riesz basis for $\mathcal{H}$ and $\{g_n\}$ is a complete sequence in $\mathcal{H}$ such that $$\ds\sum\|f_n-g_n\|<\infty,$$ then $\{g_n\}$ is a Riesz basis for $\mathcal{H}$. 
Consequently, we have the following
\begin{thm}[\textbf{Replacement Theorem}]
The Riesz basis property of the system $\mathcal{E}(X;m)$ is unaffected if we replace finitely many points $x_n$ by the same number of points $y_n\notin X$.
\end{thm}
The following four results are the main ingredients to prove our main results.
\begin{thm}[Schmidts's inequality]$($\cite{Agarwal}$)$.
If $p_n(x)$ is a polynomial of degree $\leq n$, then
\begin{eqnarray}\label{pap3eqn1.4}
\int\limits_a^b|p_n^\prime(x)|^2~\mathrm{d}x\leq \mu_n(b-a)^{-2}\int\limits_a^b|p_n(x)|^2~\mathrm{d}x,
\end{eqnarray}
where $\mu_n=\dfrac{n(n+1)(n+2)(n+3)}{2}$.
\end{thm}

\begin{thm}[Cimmino's inequality]$($\cite{Agarwal}$)$.
%Let $f$ be a complex valued  function defined on the interval $\left[a,b\right]$.
If $f$ is $r$ times continuously differentiable in $\left[a,b\right]$ with $f^{(l)}(a)=f^{(l)}(b)=0$, $0\leq l\leq r-1$, then
\begin{eqnarray}\label{eqn2.8}
\int\limits_a^b|f^{(k)}(x)|^2~\mathrm{d}x\leq \left(\dfrac{b-a}{\lambda_{r,k}}\right)^{2r-2k}\int\limits_a^b|f^{(r)}(x)|^2~\mathrm{d}x, ~0\leq k\leq r-1,
\end{eqnarray}
where $\lambda_{r,k}^{2r-2k}$ is the first eigenvalue of the boundary value problem
\begin{eqnarray*}
u^{(2r)}(x)-\lambda(-1)^{r+k} u^{(2k)}(x)=0,~\lambda>0,~~ x\in[0,1],\\
u^{(k)}(0)=u^{(k)}(1)=0, ~~0\leq k\leq r-1,~~u\in C^{2r}[0,1].
\end{eqnarray*}
In \eqref{eqn2.8} equality holds if and only if $f(x)$ is the first eigenfunction of the boundary  value problem
$u^{(2r)}(x)-\lambda(-1)^{r+k} u^{(2k)}(x)=0,~u^{(k)}(a)=u^{(k)}(b)=0, ~~0\leq k\leq r-1.$

\end{thm}
\begin{thm}$($\cite{Grochenig1},\cite{Raza}$)$.\label{thm2.6}
Let $A$ be a bounded operator on a Hilbert space $\mathcal{H}$ that satisfies
\begin{align*}
\|f-Af\|_\mathcal{H}\leq C\|f\|_\mathcal{H},
\end{align*}
for every $f\in \mathcal{H}$ and for some $C$, $0<C<1$.
Then $A$ is invertible on $\mathcal{H}$ and $f$ can be recovered from $Af$ by the following iteration algorithm.
Setting
\begin{eqnarray*}
f_0&=& A f ~and \\
f_{n+1}&=&f_n+A(f-f_n), ~n\geq 0,
\end{eqnarray*}
we have $\lim\limits_{n\to\infty}f_n=f$. The error estimate after $n$ iterations is
\begin{align*}
\|f-f_n\|_\mathcal{H}\leq C^{n+1}\|f\|_\mathcal{H}.
\end{align*}
\end{thm}
\begin{thm}[Hermite Interpolation Formula] $($\cite{Spitzbart}$)$.
Let $f$ be $r$ times continuously differentiable in $\left[a,b\right]$ and $\xi,\eta\in[a,b]$. Then the Hermite interpolation polynomial $H_{2r+1}(x)$ of degree $2r+1$ such that
$H_{2r+1}^{(j)}(y)=f^{(j)}(y)$, for $y=\xi,\eta$, $0\leq j\leq r$, is given by
\begin{eqnarray}
H_{2r+1}(\xi,\eta,f;x)=\ds\sum\limits_{k=0}^{r}A_{0k}(x)f^{(k)}(\xi)+\ds\sum\limits_{k=0}^{r}A_{1k}(x)f^{(k)}(\eta),
\end{eqnarray}
where
\begin{eqnarray*}
A_{0k}(x)&=&(x-\eta)^{r+1}\dfrac{(x-\xi)^k}{k!}\ds\sum\limits_{s=0}^{r-k}\dfrac{1}{s!}g_0^{(s)}(\xi)(x-\xi)^s,\\
A_{1k}(x)&=&(x-\xi)^{r+1}\dfrac{(x-\eta)^k}{k!}\ds\sum\limits_{s=0}^{r-k}\dfrac{1}{s!}g_1^{(s)}(\eta)(x-\eta)^s,\\
g_0(x)&=&(x-\eta)^{-(r+1)},\\
g_1(x)&=&(x-\xi)^{-(r+1)}.
\end{eqnarray*}
\end{thm}

Using Cimmino's inequality, the following error estimate 
 \begin{eqnarray}\label{errorestimate}
\int\limits_a^b|f^{(r-1)}(x)-H_{2r-1}^{(r-1)}(x)|^2~\mathrm{d}x\leq \left(\dfrac{b-a}{\nu_r}\right)^{2}\int\limits_a^b|f^{(r)}(x)|^2~\mathrm{d}x,
\end{eqnarray}
is proved in \cite{Agarwal}, where $\nu_r=\lambda_{r,r-1}$. We explicitly mention certain values for the constants $\nu_{r}$ as given in \cite{Agarwal}.
\begin{eqnarray*}
\nu_1=\pi, ~\nu_2=2\pi, ~\nu_3=8.9868.
\end{eqnarray*}

\section{A Uniqueness Theorem}
Let us define the sample sets $\Lambda_N=\{\lambda_n(l):l\in\mathbb{Z}\},~N\geq 1$ as follows:
\begin{eqnarray*}
 \lambda_n(l)&:=& \left\{\begin{array}{ccc}
\vspace{0.3cm}
\dfrac{k\pi l}{\sigma}& \mbox{\hspace{0.4cm}if $|l|>N$,}\\
\vspace{0.3cm}0& \mbox{if $l=0$,}\\
\text{sgn}{(l)}(2l-1)\dfrac{k\pi(N+1)}{\sigma(2N+1)} & \mbox{\hspace{1.2cm}if $1\leq |l|\leq N$.}
\end{array} \right.
\end{eqnarray*}
It is clear that $\Lambda_N$ is equal to $\dfrac{k\pi}{\sigma}\mathbb{Z}$ except at finite number of points. 
Then it follows from Replacement Theorem that $\Lambda_N$ is an interpolating set of order $k-1$ for $\mathcal{B}_\sigma$.
Consider the sequence $(c_\lambda)_{\lambda\in\Lambda_N}$ defined as
\begin{eqnarray*}
 c_\lambda&:=& \left\{\begin{array}{cc}
1&\hspace{-1cm}\mbox{if $\lambda=0$,}\\
0&\mbox{\hspace{0.4cm}if $\lambda\in\Lambda_N-\{0\}$.}
\end{array} \right.
\end{eqnarray*}
Then there exists a non-zero function $g_N\in\mathcal{B}_\sigma$ such that $g_N^{(l)}(\lambda)=c_\lambda, 0\leq l\leq k-1$. 
Notice that $M_k(g_N)=\dfrac{2k(N+1)\pi}{(2N+1)\sigma}$ converges to $\dfrac{k\pi}\sigma{}$ as $N\to\infty$.
Therefore, $d_k\leq\dfrac{k\pi}{\sigma}.$

\begin{thm}
If a non-zero function $f\in \mathcal{B}_\sigma$  has infinitely many zeros on the real axis of multiplicity $k$, then there exists at least one pair of consecutive zeros whose distance apart is 
greater than $\dfrac{\nu_{k}}{\sigma}$. Consequently, $\dfrac{\nu_{k}}{\sigma}\leq d_k\leq \dfrac{k\pi}{\sigma}$.
\end{thm}
\begin{proof}
Let a non-zero function $f\in \mathcal{B}_\sigma$ have infinitely many zeros $x_j$'s of multiplicity $k$ on the real line such that $x_j<x_{j+1}$, $j\in\mathbb{Z}$ and $\bigcup\limits_{j\in\mathbb{Z}}[x_j,x_{j+1}]=\mathbb{R}$.
If possible, there exists $M\leq\dfrac{\nu_{k}}{\sigma}$ such that $x_{j+1}-x_j\leq M$, for every $j$. Since $f^{(l)}(x_j)=f^{(l)}(x_{j+1})=0$, for every $j\in\mathbb{Z}$, $0\leq l\leq k-1$, by Cimmino's inequality
\begin{eqnarray}\label{pap3eqn2.3}
\ds\int\limits_{x_j}^{x_{j+1}}|f^{(k-1)}(x)|^2~\mathrm{d}x<\left(\dfrac{x_{j+1}-x_j}{\nu_{k}}\right)^2\ds\int\limits_{x_j}^{x_{j+1}}|f^{(k)}(x)|^2~\mathrm{d}x.
\end{eqnarray}
Notice that the inequality is strict; otherwise if the equality holds, then  $f(x)$ coincides with the first eigenfunction of the following boundary  value problem:
\begin{eqnarray*}
u^{(2k)}(x)+\lambda u^{(2k-2)}(x)=0,~~ x\in[x_i,x_{i+1}],\\
u^{(l)}(x_i)=u^{(l)}(x_{i+1})=0, ~~0\leq l\leq k-1,~~u\in C^{2k}[x_i,x_{i+1}].
\end{eqnarray*}
Since $f$ and the eigenfunction are entire,  $f$ coincides with the eigenfunction on the whole real axis. Moreover, the eigenfunction is a 
finite linear combination of $\{x^le^{\lambda_lx}\cos\mu_lx, x^le^{\lambda_lx}\sin\mu_lx,l=0,1,2,\dots\}$. (For example, see the case $k=2$ in \cite{Tcheng}).  Clearly, it is not square integrable on the real axis.
Hence  $f\notin L^2(\mathbb{R})$, which is impossible.
Summing over all $j$ in \eqref{pap3eqn2.3}, we get
\begin{eqnarray*}
\ds\int\limits_{\mathbb{R}}|f^{(k-1)}(x)|^2~\mathrm{d}x
&<&\ds\sum\limits_{j}\left(\dfrac{x_{j+1}-x_j}{\nu_{k}}\right)^2\ds\int\limits_{x_j}^{x_{j+1}}|f^{(k)}(x)|^2~\mathrm{d}x\\
&\leq&\left(\dfrac{M}{\nu_{k}}\right)^2\ds\int\limits_{\mathbb{R}}|f^{(k)}(x)|^2~\mathrm{d}x.
\end{eqnarray*}
Taking square root on both sides, we get
\begin{eqnarray}\label{pap3eqn2.4}
\|f^{(k-1)}\|_2<\dfrac{M}{\nu_{k}}\|f^{(k)}\|_2.
\end{eqnarray}
On the other hand, by Bernstein's inequality,
\begin{eqnarray}\label{pap3eqn2.5}
\|f^{(k)}\|_2\leq\sigma~\|f^{(k-1)}\|_2.
\end{eqnarray}
Since $f$ is nonconstant entire function which is square integrable on the real axis, $f^{(k-1)}$ is non-zero.
Combining $\eqref{pap3eqn2.4}$ and $\eqref{pap3eqn2.5}$, we get $M>\dfrac{\nu_{k}}{\sigma}$ which is a contradiction.
\end{proof}

\begin{cor}
If $X=\{x_i:i\in\mathbb{Z}\}$ is sequence of real numbers such that $\sup\limits_{i}(x_{i+1}-x_i)\leq\dfrac{\nu_{k}}{\sigma}$, then $X$ is a set of uniqueness of order $k-1$ for $\mathcal{B}_\sigma$.
\end{cor}
Since $\nu_2=\lambda_{2,1}=2\pi$, we obtain the following sharp results.
\begin{cor}
The constant $d_2$ is equal to $\dfrac{2\pi}{\sigma}$.
\end{cor}
\begin{cor}
If $X=\{x_i:i\in\mathbb{Z}\}$ is sequence of real numbers such that $\sup\limits_{i}(x_{i+1}-x_i)\leq\dfrac{2\pi}{\sigma}$, 
then  $X$ is a set of uniqueness of order $1$, \textit{i.e.,} if $f(x_i)=f^\prime(x_i)=0$, for all $i\in\mathbb{Z}$, then $f\equiv 0$.
\end{cor}

\section{Reconstruction Algorithms}
Let $\mathcal{D}$ denote the differentiation operator on $\mathcal{B}_\sigma$.  By Bernstein inequality, it is a bounded operator on $\mathcal{B}_\sigma$.
If $f\in\mathcal{B}_\sigma$, then $|f(x)|\to 0$ as $|x|\to \infty$ which implies that
$\int\limits_{-\infty}^\infty f^\prime(x)\overline{g(x)}\mathrm{d}x=-\int\limits_{-\infty}^\infty f(x)\overline{g^\prime(x)}\mathrm{d}x$. Therefore, $\mathcal{D}$ is a skew-Hermitian operator on $\mathcal{B}_\sigma.$
Let us define $\mathcal{B}_{\sigma}^{k-1}:= \mathcal{D}^{k-1}(\mathcal{B}_\sigma)$. Clearly $\mathcal{B}_\sigma^{k-1}$ is a subspace (not necessarily closed) of $L^2(\mathbb{R})$.

%Consider the operator $P:L^2(\mathbb{R})\to \overline{\mathcal{B}_{\sigma}^{k-1}}$ by
%\begin{eqnarray}\label{pap2eqn2.6}
%(Pf)(x):=\langle f^{(k-1)}, K_x\rangle,
%\end{eqnarray}
%where $K_x(t)=\dfrac{\sin \sigma\left(t-x\right)}{\sigma \left(t-x\right)}$. Then
Let $P$ be the orthogonal projection of $L^2(\mathbb{R})$ onto $\overline{\mathcal{B}_{\sigma}^{k-1}}$.
Now assume that $f$ and its first $k-1$ derivatives $f', \dots,f^{(k-1)}$ are sampled at a sequence $(x_i)_{i\in\mathbb{Z}}$.
Define the approximation operator for $f\in \mathcal{B}_{\sigma}$
\begin{align*}
A[f^{(k-1)}]=P\left(\left[\ds\sum\limits_{i\in\mathbb{Z}}\left(H_{2k-1}^{(k-1)}(x_i,x_{i+1},f;\cdot)\right)\chi_{[x_i,x_{i+1}]}\right]\right),
\end{align*}
where $H_{2k-1}(x_i,x_{i+1},f;\cdot)$ denotes the Hermite interpolation of $f$ in the interval $[x_i,x_{i+1}]$.
Since  $f^{(k-1)}=Pf^{(k-1)}=P\left(\ds\sum\limits_{i\in\mathbb{Z}}f^{(k-1)}\chi_{[x_i,x_{i+1}]}\right)$ for all $f\in \mathcal{B}_\sigma$ and the characteristic functions $\chi_{[x_i,x_{i+1}]}$ 
have mutually disjoint support, it can be easily shown that
\begin{eqnarray*}
\left\|f^{(k-1)}-A[f^{(k-1)}]\right\|_2^2&\leq&\ds\sum\limits_{i\in\mathbb{Z}}\int\limits_{x_i}^{x_{i+1}}|f^{(k-1)}(x)-H_{2k-1}^{(k-1)}(x_i,x_{i+1},f;x)|^2~\mathrm{d}x.
\end{eqnarray*}
If $\sup\limits_i(x_{i+1}-x_i)=\delta$, then it follows from \eqref{errorestimate} that
\begin{eqnarray}\label{eqn4.1}
\left\|f^{(k-1)}-A[f^{(k-1)}]\right\|_2^2
&\leq&\ds\sum\limits_{i\in\mathbb{Z}}\left(\dfrac{\delta}{\nu_{k}}\right)^2
\int\limits_{x_i}^{x_{i+1}}|f^{(k)}(x)|^2~\mathrm{d}x\nonumber\\
&=&\left(\dfrac{\delta}{\nu_{k}}\right)^2\|f^{(k)}\|_2^2\nonumber\\
&\leq&\left(\dfrac{\delta}{\nu_{k}}\right)^2\sigma^2\|f^{(k-1)}\|_2^2,
\end{eqnarray}
using Bernstein's inequality. As $\|Af^{(k-1)}\|_2\leq \|f^{(k-1)}-Af^{(k-1)}\|_2+\|f^{(k-1)}\|_2$, it follows from the inequality \eqref{eqn4.1} 
that the operator $A$ is  bounded on $\mathcal{B}_\sigma^{k-1}$. Hence the operator $A$ can be uniquely extended to a bounded operator $\widetilde A$ on $\overline{\mathcal{B}_\sigma^{k-1}}$ such that
\begin{eqnarray}
\|g-\widetilde Ag\|_2
&\leq&\left(\dfrac{\delta\sigma}{\nu_{k}}\right)\|g\|_2,
\end{eqnarray}
for all $g\in \overline{\mathcal{B}_\sigma^{k-1}}$.
If $\delta<\dfrac{\nu_{k}}{\sigma}$, then $\dfrac{\delta\sigma}{\nu_{k}}<1$ which implies that the operator $\widetilde A$ is invertible on $\overline{\mathcal{B}_\sigma^{k-1}}$.
Thus  we can obtain the following result as a corollary of Theorem \ref{thm2.6}.
\begin{thm}\label{pap3thm2.2}
Suppose that $f$ and its first $k-1$ derivatives $f', \dots,f^{(k-1)}$ are sampled at a sequence $(x_i)_{i\in\mathbb{Z}}$. 
If $\delta<\dfrac{\nu_{k}}{\sigma}$, then  any $f\in \mathcal{B}_\sigma$ can be reconstructed from the sample values 
$\{f^{(j)}(x_i)F:j=0,1,\dots,k-1,i\in\mathbb{Z}\}$ using the following iteration algorithm.
Set
\begin{eqnarray*}
f_0&=&A f^{(k-1)}=P\left(\ds\sum\limits_{i\in\mathbb{Z}}H_{2k-1}^{(k-1)}(x_i,x_{i+1},f;\cdot)\chi_{[x_i,x_{i+1}]}\right),\\
f_{n+1}&=&f_n+\widetilde A(f^{(k-1)}-f_n),~ n\geq 0,
\end{eqnarray*}
where $H_{2k-1}(x_i,x_{i+1},f;\cdot)$ denotes the Hermite interpolation of $f$ in the interval $[x_i,x_{i+1}]$.
Then we have $\lim\limits_{n\to\infty}f_n=f^{(k-1)}$.
The error estimate after $n$ iterations becomes
\begin{eqnarray*}
\|f^{(k-1)}-f_n\|_2&\leq&\left(\dfrac{\delta\sigma}{\nu_{k}}\right)^{(n+1)}\|f^{(k-1)}\|_2.
\end{eqnarray*}

\end{thm}

As a consequence of the above theorem, we construct a frame (or atomic system) for differentiation operator on $\mathcal{B}_\sigma$.
Let $c_{i,l}=\ds\int\limits_{x_i}^{x_{i+1}}\dfrac{(x-x_{i+1})^{2l}}{l!^2}~\mathrm{d}x$. This can also be written as
\begin{eqnarray*}
c_{i,l}=\dfrac{(x_{i+1}-x_i)^{2l+1}}{(2l+1)l!^2}=\ds\int\limits_{x_i}^{x_{i+1}}\dfrac{(x-x_{i})^{2l}}{l!^2}~\mathrm{d}x.
\end{eqnarray*}
Let $X$ be separated by a constant $\gamma>0$, \textit{i.e.,} assume that $\inf\limits_{i\neq j} |x_i-x_j|\geq\gamma$.
It is proved in \cite{AntoRad2} that
$\dfrac{1}{2kC(k)}\ds\sum\limits_{i\in\mathbb{Z}}\int\limits_{x_i}^{x_{i+1}}|H_{2k-1}(x_i,x_{i+1},f;x)|^2~\mathrm{d}x$
\begin{eqnarray}\label{eqn4.3}
&\leq&\ds\sum\limits_{i\in\mathbb{Z}}
\ds\sum\limits_{l=0}^{k-1}|f^{(l)}(x_i)|^2(c_{i,l}+c_{i-1,l})\leq B\| f\|_2^2,
\end{eqnarray}
where
$B=2\left(\ds\sum\limits_{l=0}^{k-1}\dfrac{(\delta\sigma)^{2l}}{l!^2}\right)e^{\delta^2+\sigma^2}$ and $C(k)=\left[\ds\sum\limits_{s=0}^{k-1}{k+s-1\choose s} \right]^2$.
%Let $g_0(x)=(x-\eta)^{-k}$ and $g_1(x)=(x-\xi)^{-k}$. Then
%\begin{eqnarray*}
%g_0^{(s)}(x)&=&(-1)^s k (k+1)\dots(k-1+s)(x-\eta)^{-(k+s)}~\textrm{and}\\
%g_1^{(s)}(x)&=&(-1)^s k (k+1)\dots(k-1+s)(x-\xi)^{-(k+s)}.
%end{eqnarray*}

\begin{thm}\label{thm4.2}
If $X=\{x_i\}$ is a separated set such that $\delta<\dfrac{\nu_{k}}{\sigma}$, then for every $f\in \mathcal{B}_\sigma$, we have
\begin{eqnarray}\label{eqn4.4}
A\| \mathcal{D}^{k-1}f\|_2^2&\leq&\ds\sum\limits_{i\in\mathbb{Z}}
\ds\sum\limits_{l=0}^{k-1}|f^{(l)}(x_i)|^2(c_{i,l}+c_{i-1,l})\leq B\| f\|_2^2,
\end{eqnarray}
where
$A=\left(1-\dfrac{\delta\sigma}{\nu_k}\right)^{2}\dfrac{1}{2k C(k)}\dfrac{\gamma^{2(k-1)}}{\mu_{2k-1}^{k-1}}$,
$B=2\left(\ds\sum\limits_{l=0}^{k-1}\dfrac{(\delta\sigma)^{2l}}{l!^2}\right)e^{\delta^2+\sigma^2}$, and\\
$C(k)=\left[\ds\sum\limits_{s=0}^{k-1}{k+s-1\choose s} \right]^2$,
\textit{i.e.,} $X$ is a stable set of sampling of order $(k-1,k-1)$ for $\mathcal{B}_\sigma$ with respect to the weight $\{c_{il}+c_{i-1,l}:l=0,\dots,k-1,i\in\mathbb{Z}\}$.
\end{thm}
\begin{proof}
Recall $\widetilde{A}f^{(k-1)}=Af^{(k-1)}=P\left(\ds\sum\limits_{i\in\mathbb{Z}}H_{2k-1}^{(k-1)}(x_i,x_{i+1},f;\cdot)\chi_{[x_i,x_{i+1}]}\right)$. Then
\begin{eqnarray}\label{pap3eqn2.8}
\|f^{(k-1)}\|_2^2&=&\|\widetilde{A}^{-1}\widetilde{A}f^{(k-1)}\|_2^2\nonumber\\
&\leq&\|\widetilde{A}^{-1}\|^2\|Af^{(k-1)}\|_2^2\nonumber\\
&\leq&(1-\|I-\widetilde{A}\|)^{-2}\|Af^{(k-1)}\|_2^2\nonumber\\
&\leq&\left(1-\dfrac{\delta\sigma}{\nu_k}\right)^{-2}\|Af^{(k-1)}\|_2^2.
\end{eqnarray}
We now estimate the value of $\|Af^{(k-1)}\|_2$.
\begin{eqnarray}\label{eqn4.6}
\|A f^{(k-1)}\|_2^2
&\leq&\left\|\ds\sum\limits_{i\in\mathbb{Z}}H_{2k-1}^{(k-1)}(x_i,x_{i+1},f;\cdot)\chi_{[x_i,x_{i+1}]}\right\|_2^2\nonumber\\
&=&\ds\int\limits_{\mathbb{R}}\left|\ds\sum\limits_{i\in\mathbb{Z}}H_{2k-1}^{(k-1)}(x_i,x_{i+1},f;x)\chi_{[x_i,x_{i+1}]}(x)\right|^2~\mathrm{d}x\nonumber\\
&\leq&\ds\sum\limits_{i\in\mathbb{Z}}\int\limits_{x_i}^{x_{i+1}}|H_{2k-1}^{(k-1)}(x_i,x_{i+1},f;x)|^2~\mathrm{d}x\nonumber\\
&\leq&\dfrac{\mu_{2k-1}^{k-1}}{\gamma^{2(k-1)}}\ds\sum\limits_{i\in\mathbb{Z}}\int\limits_{x_i}^{x_{i+1}}|H_{2k-1}(x_i,x_{i+1},f;x)|^2~\mathrm{d}x,
\end{eqnarray}
by Schmidt's inequality. Hence the desired inequality \eqref{eqn4.4} follows from \eqref{eqn4.3} and \eqref{eqn4.6}.
\end{proof}
\begin{cor}\label{cor4.1}
 If $\{x_i\}$ is a separated set such that $\delta<\dfrac{\nu_{k}}{\sigma}$, then there exists a Bessel sequence $\{g_{i,l}:i\in\mathbb{Z},l=0,\dots,k-1\}$ in $\mathcal{B}_\sigma$ such that
 $$f^{(k-1)}(x)=\sum\limits_{i\in\mathbb{Z}}\sum\limits_{l=0}^{k-1} f^{(l)}(x_i)g_{i,l}(x),$$
 for every $f\in \mathcal{B}_\sigma$.
 The function $f$ can be computed using the relation $$f^{(l-1)}(x)=\int\limits_{x_i}^x f^{(l)}(t)\mathrm{d}t+f^{(l-1)}(x_i), ~l=k-1, k-2, \dots, 1.$$
\end{cor}
\begin{proof}
 Since $\mathcal{D}$ is a skew-Hermitian operator on $\mathcal{B}_\sigma$ and $f^{(r)}(x)=(-1)^r\langle f, K_x^{(r)} \rangle$, the result follows from Theorem \ref{thm2.1}.
\end{proof}

We know that $\nu_2=2\pi$. Consequently, we obtain our  Theorem \ref{thm1.1} as a particular case of the above corollary. We do not know how to find the Bessel sequence given in Corollary \ref{cor4.1}.
Since the differentiation operator is not bounded below, we cannot apply the frame reconstruction algorithm also. However, we are going to apply the frame algorthim for a large subclass of $\mathcal{B}_\sigma.$

For given $\epsilon>0$, we introduce a closed subspace in $L^2(\mathbb{R})$
$$\mathcal{B}_{\sigma,\epsilon}:=\left\{f\in L^2(\mathbb{R}): \text{supp}(\widehat{f})\subseteq \left[-\frac{\sigma}{2\pi},-\epsilon\right]\cup\left[\epsilon, \frac{\sigma}{2\pi}\right]\right\}.$$
Clearly $\mathcal{B}_{\sigma,\epsilon}\subseteq \mathcal{B}_\sigma.$  
Since $\|f^\prime\|=\|\widehat{f^\prime}\|=\|2\pi iw\widehat{f}\|\geq2\pi\epsilon\|f\|$, the differentiation operator is bounded below on $\mathcal{B}_{\sigma,\epsilon}$.
Consequently, it follows from Theorem 4.2 that  if $\delta<\dfrac{\nu_{k}}{\sigma}$, there exist $A_\epsilon,B>0$ such that
\begin{eqnarray}\label{pap3eqn2.7}
A_\epsilon\|f\|_2^2&\leq&\ds\sum\limits_{i\in\mathbb{Z}}
\ds\sum\limits_{l=0}^{k-1}|f^{(l)}(x_i)|^2(c_{i,l}+c_{i-1,l})\leq B\| f\|_2^2,
\end{eqnarray}
for every $f\in \mathcal{B}_{\sigma,\epsilon}$. 
Recall that $\mathcal{B}_\sigma$ is a reproducing kernel Hilbert space with reproducing kernel $K(x,y)=\dfrac{\sin \sigma\left(x-y\right)}{\sigma \left(x-y\right)}$.  
\textit{i.e.,} every $f\in \mathcal{B}_\sigma$ can be written as
\begin{eqnarray}
f(x)=\int\limits_{\mathbb{R}}f(t)\dfrac{\sin \sigma\left(t-x\right)}{\sigma \left(t-x\right)}~dt
=\langle f, K_x \rangle,
\end{eqnarray}
where $K_x(t)=K(t,x)$. Moreover, $f^{(r)}(x)=(-1)^r\langle f, K_x^{(r)} \rangle$. Hence, if $\delta<\dfrac{\nu_{k}}{\sigma}$,
it follows from Theorem \ref{thm4.2} that the family $\{\sqrt{c_{i,l}+c_{i-1,l}}K_{x_i}^{(l)}:l=0,1,\dots, k-1, i\in\mathbb{Z}\}$ is a frame for $\mathcal{B}_{\sigma, \epsilon}$ with frame bounds $A_\epsilon$ and $B$.
This leads us to the following reconstruction algorithm for $f\in\mathcal{B}_{\sigma,\epsilon}$ from its samples values.

\noindent{}
\textsf{\underline{Frame Algorithm:}}

Set $S_kf:=\ds\sum\limits_{i\in\mathbb{Z}}
\ds\sum\limits_{l=0}^{k-1}(-1)^lf^{(l)}(x_i)(c_{i,l}+c_{i-1,l})K_{x_i}^{(l)}$
and $\rho=\dfrac{2}{A_\epsilon+B}$. Define
\begin{eqnarray*}
f_0&=&0,\\
f_{n+1}&=&f_n+\rho S_k(f-f_n),~ n\geq 0.
\end{eqnarray*}
Then we have $\lim\limits_{n\to\infty}f_n=f$. The error estimate after $n$ iterations turns out to be
\begin{eqnarray*}
\|f-f_n\|_2&\leq&\left(\dfrac{B-A_\epsilon}{B+A_\epsilon}\right)^n\|f\|_2.
\end{eqnarray*}

\end{document}